\documentclass[12pt]{amsart} 
\usepackage{color}
\usepackage[english]{babel} % English

\usepackage{amsmath,amsfonts,amsthm} % Math packages
\usepackage{fancyhdr} % Custom headers and footer
\usepackage{verbatim} %Allows to use comments
\usepackage{hyperref} %Hyperlink reference
\hypersetup{          %Set up  for hyperlinks
    colorlinks=true, breaklinks, linkcolor=[rgb]{0.15 0 0.8}, filecolor=magenta, urlcolor=blue,
    citecolor=magenta, linktoc=all, }
\usepackage[nameinlink]{cleveref}
\crefformat{enumi}{#2\textup{#1}#3}
\usepackage{bm} 
\usepackage{graphicx} %This is for inserting figures and captions
\usepackage[font=footnotesize,labelfont=bf]{caption} %Changes the font type from
\usepackage{tikz-cd} %Easy commutative diagrams

\usepackage{float} %positioning of figures

\numberwithin{equation}{section} % Number equations within sections (i.e. 1.1,
\numberwithin{figure}{section} % Number figures within sections (i.e. 1.1, 1.2,
\numberwithin{table}{section} % Number tables within sections (i.e. 1.1, 1.2,
\title{Positive factorizations of pseudoperiodic homeomorphisms} % Thick bottom
\author{Pablo Portilla Cuadrado} % Name of the authors
\thanks{The author is supported by CONACYT project wtih No. 286447}

\date{\normalsize\today} % Today's date or a custom date

\usepackage{enumitem}

\usepackage{aliascnt}

\usepackage[all]{xy}

\usepackage{pifont}

\usepackage{setspace}
\usepackage[refpage,intoc]{nomencl}

\newtheorem{thm}[equation]{Theorem}
\Crefname{thm}{Theorem}{thm}

\newtheorem{lemma}[equation]{Lemma}
\newtheorem{prop}[equation]{Proposition}
\Crefname{prop}{Proposition}{thm}
\newtheorem{cor}[equation]{Corollary}
\theoremstyle{definition}

\newtheorem{definition}[equation]{Definition}
\newtheorem{notation}[equation]{Notation}
\Crefname{notation}{Notation}{notation}
\newtheorem{remark}[equation]{Remark}
\newcommand{\RR}{\mathbb{R}}
\newcommand{\QQ}{\mathbb{Q}}
\newcommand{\C}{\mathbb{C}}
\newcommand{\NN}{\mathbb{N}}

\newcommand{\ZZ}{\mathbb{Z}}

\newcommand{\Si}{\Sigma}

\newcommand{\id}{\mathop{\mathrm{id}}\nolimits}

\newcommand{\calB}{\mathcal{B}}
\newcommand{\calC}{\mathcal{C}}

\newcommand{\calX}{\mathcal{X}}

\newcommand{\dehn}{\mathrm{Dehn}^+_{g,r}}
\newcommand{\veer}{\mathrm{Veer}_{g,r}}
\newcommand{\modgr}{\mathrm{Mod}_{g,r}}
\newcommand{\pmodgr}{\mathrm{PMod}_{g,r}}
\newcommand{\fr}{\mathrm{fr}}
\newcommand{\sn}{\mathrm{sc}}
\usepackage{mathtools}
\DeclarePairedDelimiter\floor{\lfloor}{\rfloor}
\DeclarePairedDelimiter{\ceil}{\lceil}{\rceil}

\newcommand{\htt}{\hat{t}}

\usepackage[bindingoffset=0in,left=1.3in,right=1.3in,top=1.5in,bottom=1.5in,footskip=.85in]{geometry}

\begin{document}
\maketitle

\begin{abstract}
We generalize a classical result concerning smooth germs of surfaces, by proving that monodromies on links of isolated complex surface singularities associated with reduced holomorphic map germs admit a positive factorization. As a consequence of this and a topological characterization of these monodromies by Anne Pichon, we conclude that a pseudoperiodic homeomorphism on an oriented surface with boundary with positive fractional Dehn twist coefficients and screw numbers, admits a positive factorization.  We use the main theorem to give a sufficiency criterion for certain pseudoperiodic homeomorphisms with negative screw numbers to admit a positive factorization.
\end{abstract}

\tableofcontents

\section{Introduction}

Although the main result of this paper has its greatest impact in the theory of mapping class groups, its original motivation lies in the field of singularity theory. We describe this motivation first.

 Consider a reduced germ of holomorphic function $f:(\C^2,0) \to (\C,0)$. Therefore, it has a representative defined on a open neighborhood $U$ of $0$ in $\C^2$, with a unique critical point at $0$. For simplicity, let us still denote by $f$ the restriction of this representative to a compact neighborhood $B$ of $0$ in $U$.  Then for a generic linear form $\ell:\C^2 \to \C$ and  a non-zero real number $\epsilon'$ with $|\epsilon'|$ small enough, we have that $$\tilde{f}:=f + \epsilon' \ell: B \to \C$$ is what is usually known as a {\bf morsification} of $f$. That is, a holomorphic map that only has Morse-type singularities with distinct critical values all of them close to $0 \in \C$. On another hand, we have that \begin{equation}\label{eq:locally_trivial_milnorle}
f_{|f^{-1}(\partial D_\delta) \cap B_{\epsilon}}: f^{-1}(\partial D_\delta) \cap B_{\epsilon}  \to \partial D_\delta \end{equation} is a locally trivial fibration for $B_\epsilon \subset \C^2$ a closed ball of small radius $\epsilon>0$ and $ D_\delta \subset \C$ a disk of small  radius $0<\delta<<\epsilon$. This is known as the Milnor fibration on the tube \cite{Mil}.

Since transversality is an open condition, $$\tilde{f}_{| \tilde{f}^{-1}(\partial D_\delta) \cap B_{\epsilon}}: \tilde{f}^{-1}(\partial D_\delta) \cap B_{\epsilon}  \to \partial D_\delta$$ defines a fibration equivalent to \cref{eq:locally_trivial_milnorle} for $\epsilon'$ small enough. We can assume $B\subset B_\epsilon$ so that $\tilde{f}$ has only isolated critical points near the origin $0\in\C^2$. Therefore one can easily deduce that the monodromy associated with $f$ can be written as a composition of right-handed Dehn twists: it admits a positive factorization. This classical line of argument is known as Picard-Lefschetz theory and works well more generally when the ambient space is a  smooth Stein manifold.

Let us replace $\C^2$ by a singular ambient space. Take $X$ to be an isolated complex surface singularity. Let $f:X\to \C$ be a reduced element of the maximal ideal of its local ring of germs of holomorphic functions. Associated with $f$ there is also a locally trivial fibration $$f_{|f^{-1}(\partial D_\delta) \cap X \cap B_\epsilon}: f^{-1}(\partial D_\delta) \cap X \cap B_\epsilon \to \partial D_\delta$$ on a {\em tube}. This one is known as the Milnor-L\^e fibration \cite{Le}. Here $B_\epsilon$ denotes a Milnor ball for $X$, that is, a compact Euclidean ball centered at the base point $x$ of $X$ in an ambient space $\C^p$ of a representative of it, such that all Euclidean spheres contained in this ball and centered at $x$ are transversal to the representative. 

Now we observe that a similar game does not work since a small perturbation $f + \epsilon'\ell$ has some Morse singularities on the smooth part of $ (f+\epsilon'\ell)^{-1}(\partial D_\delta) \cap X \cap B_\epsilon$ {\em and} the singularity defined by $\ell$ at the singular point of $X$. In order to learn more about this topic, one may consult Siersma's paper \cite{Sier} and Tib\v{a}r's paper \cite{Tib}. Therefore, by these methods, we cannot conclude that the monodromy defined by $f$ can be written as a composition of right-handed Dehn twists. To decide if these monodromies admit (or not) positive factorizations is the original motivation of this paper.

We argue as follows to solve this problem. First we consider the minimal resolution $\tilde{X} \to X$ of $X$. This gives us a strongly pseudoconvex surface $\tilde{X}$ with a non-trivial exceptional set without curves of the first kind. For this kind of surfaces, Bogomolov and De Oliveira \cite{Bog} based on previous work by Laufer \cite{Lau,Lau2}, developed methods to prove the existence of a deformation $\omega:\calX \to Q$ of $X$ over a disk $Q$ centered at $0 \in \C$, such that all fibers above points of $Q \setminus \{0\}$ are Stein. As a $C^\infty$-fibration, this deformation is a trivial fibration. We lift $f$ to a holomorphic map on the resolution $\tilde{f}:\tilde{X} \to \C$. Since $\calX$ has the structure of a trival $C^\infty$-fibration, we may extend $\tilde{f}$ from the central fiber to all the fibers $X_t$ of the fibration as complex valued (but not holomorphic) smooth maps $\tilde{f}_t:X_t \to \C$. Here $\tilde{f}_t$ depends smoothly on $t$.

Now we ask the question: Can we {\em correct} $\tilde{f}_t$ so that it becomes holomorphic? Note that $\bar{\partial}\tilde{f}_t$ is small in fibers close to the central fiber of the deformation. We need to solve what is classically known as the $\mathit{\bar{\partial}}$-{\em problem} and get a nice enough bound for the solution. More concretely, we look for smooth functions $u_t:X_t \to \C$ such that $$\bar{\partial} u_t = \bar{\partial} \tilde{f}_t$$ and such that $u_t$ takes small values and its first partial derivatives also take small values. More concretely we want $u_t$ to vary continuously with $t$ in the $C^1$ topology. It turns out that the PDE methods initiated by Kohn and Nirenberg \cite{Kohn} and followed by H\"ormander \cite{Hor} are not enough. These methods give only bounds of the $L^2$-norm of $u$, which does not assure us that $u$ takes small values. A few years later, Henkin \cite{Hen}  developed  {\em integral representation} methods to improve these results. His work was later generalized by Kerzman's paper \cite{Ker} to strongly pseudoconvex subdomains of Stein manifolds.  These {\em explicit} integral solutions and generalizations by Kohn and Folland \cite{Foll} of previous work by Kohn  \cite{KohnI,KohnII} allowed Greene and Krantz to produce stability results for families of strongly pseudoconvex manifolds where the complex structure varies smoothly with the parameter of deformation \cite[Sections 3 and 4]{Greene}. These are the results that we need.

In the end we are able to construct such $u_t:X_t \to \C$ with $\tilde{f}_t-u_t$ a holomorphic map reduced on the boundary of $X_t$ which defines, in restriction to $(\tilde{f}_t-u_t)^{-1}(\partial D_\delta) \cap X_t \cap B_\epsilon$ a locally trivial fibration equivalent to the one defined by the original function $f$ on $f^{-1}(\partial D_\delta)\cap X \cap B_\epsilon$. Since $X_t$ is Stein, it does not contain smooth compact analytic curves without boundary and so we can apply Picard-Lefschetz theory to conclude that our original monodromy admits a positive factorization. This is what we prove in \Cref{thm:positive_factorization} which is the main result of this paper.

In \cite{Pich}, A. Pichon proved, using a previous result by Winters \cite{Win}, a purely topological characterization of the monodromies that appear in links of isolated complex surface singularities associated with reduced holomorphic map germs. This class coincides with the class of pseudoperiodic diffeomorphisms of surfaces with boundary $\phi:\Si \to \Si$  which admit a power $\phi^n$ that is a composition of powers of right handed Dehn twists around disjoint simple closed curves including all boundary components ({\em \`a torsades n\'egatives} in \cite{Pich}). Therefore, proving that these monodromies admit positive factorizations has highly unexpected consequences in the theory of mapping class groups. As a direct corollary we find, for example, that all freely periodic diffeomorphisms with positive fractional Dehn twist coefficients automatically admit positive factorizations. Observe that Honda, Kazez and Mati\'c proved this fact in \cite{KoII} for the punctured torus.

Next we describe the organization of the paper.

\subsubsection{Outline of the paper} We start in \cref{sec:complex_geometry} by recalling some theory and definitions about pseudoconvex manifolds and their deformations. Then we discuss solutions to the $\bar{\partial}$-problem on different settings. We introduce the several norms that we use to properly state the results. After this we give a brief account of different solutions to this problem following the work of Henkin, Kerzman and Kohn. Finally we state a version of two {\em stability} theorems (\Cref{thm:continuous_variation} and \Cref{thm:kohn_stability} in this work) by Greene and Krantz that give a solution to the $\bar{\partial}$-problem for $(0,1)$ forms in smooth families of strongly pseudoconvex manifolds.

Next, in \cref{sec:mapping_class_groups}, we fix notations and conventions about certain topics of the theory of mapping class groups. This is always a necessary step if one wants to make precise statements because conventions highly vary from one author to another. We focus on the notions of fractional Dehn twist coefficient and screw numbers associated with a pseudoperiodic homeomorphism of a surface. We end the section by stating Anne Pichon's characterization of monodromies associated with reduced holomorphic map germs on links of isolated complex surface singularities.

In \cref{sec:main_thm} we start by proving a transversality proposition and then apply the theory developed in the previous sections to the proof of our main result (\Cref{thm:positive_factorization}) following the reasoning described above.

Finally, \cref{sec:consequences} is devoted to the exploration of the consequences of our main result. We start by exploiting a result by Baykur, Monden and Van Horn-Morris in \Cref{prop:pos_fact} to prove that the pure mapping class group (homeomorphisms up to isotopy free on the boundary) is generated as a semi-group by positive Dehn twists except in a few degenerate cases. This leads us to the definition of two invariants (\Cref{def:correcting,def:correcting_semigroup}) that measure the failure of a pseudoperiodic homeomorphism to be in $\dehn$. We use \Cref{prop:pos_fact} together with \Cref{thm:positive_factorization} in order to give  in \Cref{thm:criterion} a sufficiency criterion, under  suitable very general hypotheses for a pseudoperiodic homeomorphism to admit a positive factorization. The criterion roughly says that if the fractional Dehn twist coefficients are big enough with respect to the screw numbers of  certain invariant orbits of curves of the Nielsen-Thurston decomposition of the homeomorphism, then the given pseudoperiodic homeomorphism admits a positive factorization.

\section*{Acknowledgments}

I wish to thank Xavier G\'omez-Mont for inspiring conversations. I am also very thankful to Mohammad Jabbari who is an expert in the $\bar{\partial}$-problem
and gave me many useful references that lead me to find the Greene and Krantz result that I ended up using.

Thanks to Baldur Sigur{\dh}sson who read carefully an early version of
this manuscript and pointed out a gap in a lemma that was placed instead of current \Cref{prop:extension_to_family}. His critics and comments have helped me greatly improve the final manuscript.

Finally I thank the very thorough review made by one of the referees whose numerous remarks helped me to improve many parts of the article as well as to fix some proofs that were insufficiently explained or contained gaps in an early version of this article.

\section{Preliminaries on complex geometry and singularity theory}\label{sec:complex_geometry}

We start by reviewing some theory on pseudoconvex complex manifolds. Let $X$ be a complex manifold. Denoting its tangent bundle by $TX$ we can see its complex structure as an endomorphism $$J:TX \to TX$$ satisfying $J^2=-\id$ and an integrability condition.  Let $\rho:X \to \RR$ be a smooth real valued function. Let $d\rho$ denote the  exterior derivative of $\rho$. The complex structure $J$ allows us to define the complex exterior derivative as $d^{\C} \rho:= d\rho\circ J$. Let $\partial$ and $\bar{\partial}$ be the complex and complex conjugate parts of the exterior derivative. We can write $d= \partial + \bar{\partial}$ and $d^{\C}= i(\partial - \bar{\partial})$.

 We define a $(1,1)$-form $\omega_\rho :=  -dd^\C \rho$, a symmetric bilinear form $g_\rho := \omega_\rho(\cdot, J\cdot)$ and a Hermitian form $h_\rho= g_\rho - i \omega_\rho $. Under suitable circumstances described in the next definition these turn, respectively, into a symplectic form, a Riemannian metric and a Hermitian metric.

\begin{definition}\label{def:spsh_exh}
	Let $X$ be a complex manifold and let $A\subset X$ be an open subset. We say that a smooth function $\rho:X \to \RR$ is {\bf strictly plurisubharmonic} (abbreviated {\bf spsh}) on $A$ if the  symmetric field $g_\rho$ is positive-definite (that is, a Riemannian metric) on $A$. If $g_\rho$ is positive-definite on all $X$, we simply say that $\rho$ is strictly plurisubharmonic.
	
	We say that a smooth function $\rho:X \to \RR$ is an {\bf exhaustion function} if it is proper and bounded from below.
\end{definition}

This allows to define the following classical special kinds of manifolds:

\begin{definition}\label{def:spc_stein}
	We say that a complex manifold $X$ is {\bf strongly pseudoconvex} if it admits an exhaustion function $\rho:X \to \RR$ that is spsh outside a compact set. We say that $X$ is a {\bf Stein manifold} if it admits an exhaustion function which is spsh on all $X$.
	
	We say that $\bar X$ is a {\bf strongly pseudoconvex manifold with boundary} if it is a compact complex manifold with smooth boundary that admits an exhaustion function $\rho: \bar X \to \RR$ which is spsh outside a compact analytic set and such that $\partial \bar X$ is the level set  $\{x \in \bar X:\rho = 0\}$ of $\rho$. We say that $\bar X$ is a {\bf Stein domain} if it is a compact complex manifold with smooth boundary that admits an exhaustion spsh function $\rho: \bar X \to \RR$ such that $\partial \bar X = \{x \in \bar X:\rho = 0\}$.
	
	If a complex analytic space $\bar Y$ with boundary $\partial \bar Y$ satisfies that its boundary is the level set of a spsh function defined on a neighborhood of it, we say  that $\bar Y$ has a {\bf strongly pseudoconvex boundary} or that its boundary is strongly pseudoconvex.
\end{definition}

Observe that if $X$ is a Stein manifold with spsh exhaustion function $\rho:X \to \RR$ and $\{x \in X: \rho = 0\}$ is a regular level set, then $\bar X = \{x \in X: \rho (x) \leq  0\}$ is a Stein domain by definition. Moreover, every Stein domain can be obtained like this since a compact analytic manifold with smooth strongly pseudoconvex boundary can be seen as a complex submanifold of a slightly larger open complex manifold (see \cite[Remark 5.57]{Cie} and \cite{Cat}) and since the property of being strictly plurisubharmonic is stable.

\begin{remark}\label{rem:difference}
	The difference between Stein manifolds and strongly pseudoconvex manifolds is that the latter might have a non-trivial exceptional set, that is, a maximal compact complex analytic subset. And that is the only difference: a strongly pseudoconvex manifold that does not contain compact analytic sets of dimension greater than $0$ is a Stein manifold. A typical example of a strongly pseudoconvex manifold that is not a Stein manifold is the resolution space of a Milnor representative of an isolated singularity, that is, the intersection of a representative of it with an associated Milnor ball.
\end{remark}

We also recall the following alternative definition of strongly pseudoconvex manifolds with boundary:

\begin{definition}[Alternative for strongly pseudoconvex manifolds]\label{def:equivalence_strongly}
	A {\bf strongly pseudoconvex manifold with boundary} is a complex manifold with boundary such that for every point  $z \in  \partial \bar{X}$ there is a compact neighbourhood $U$ of $z$, a strongly pseudoconvex domain $D$ in $\C^n$ and a diffeomorphism $\psi$ from $U$ to $\psi(U)$ such that $\partial \psi (U) \subset \partial D$ and $\psi$ is a biholomorphism from the interior of $U$ to the interior of $\psi(U)$. This is sometimes quoted as ``the points at the boundary of a strongly pseudoconvex manifold with boundary can be presented as strongly pseudoconvex domains in $\C^n$''.
\end{definition}

The above alternative definition is taken from \cite[Definition 3.1 ]{Mari} and that it is implied by the classical definition follows from the Newlander-Nirenberg theorem for complex manifolds with boundary (see in particular \cite[Proposition 1.1]{Cat}). 

\subsection*{Preliminaries about the $\bar{\partial}$-problem}
For each pair of non negative integers $p,q \geq 0$ we denote by $C^\infty_{p,q}(X)$ the space of global differential forms of type $(p,q)$ with coefficients smooth complex valued functions. Then the operator $$\bar{\partial}: C^\infty_{p,q}(X) \to C^\infty_{p,q+1}(X)$$ acts as the complex conjugate part of the exterior derivative and satisfies $\bar{\partial}^2 = 0$.

 Observe that  $C^\infty_{0,0}(X)$ coincides with the set of smooth complex valued functions and that for an element $f \in C^\infty_{0,0}(X)$ the equation $\bar{\partial}f=0$ is satisfied precisely when $f$ is a holomorphic function.
 
 We denote by $C^\infty_p(X)$ space of the differential $p$-forms having smooth complex valued functions as coefficients.

Given an element $g \in C^\infty_{p,q+1}(X)$, it is a classical problem to determine the existence (or lack thereof) of solutions $u \in C^\infty_{p,q}(X)$ to the equation $$\bar{\partial} u = g$$ and to bound the value of some norm on $u$ by the value of some norm on $g$. This problem is known as the $\mathit{\bar{\partial}}$-{\em problem}.

In this work we wish to control the growth of the partial derivatives of {\em a} solution $u$ 
under small smooth perturbations of the complex structure $J$. In order to make sense out of the previous sentence, 
it is necessary to introduce several topologies that are used in the {\em stability} theorems that are invoked
related to this question.  We follow the notations of the article \cite{Greene}. 

Let $D \subset \C^n$ be an open set. For any continuous complex valued function $u:D \to \C$ 
we define the {\bf supremum norm} by
\begin{equation} \label{eq:sup_norm}
    ||u||_\infty:=\sup_{z \in D} |u(z)|
\end{equation} 
where $|\cdot|$ denotes the usual complex modulus. The next is called the {\bf $\bm{C^j}$ norm} and measures the size of the partial
derivatives of $u$ up to the $j$-th order. It is naturally defined for all $C^j$-functions and can possibly be infinite.
\begin{equation}\label{eq:cj_norm}
||u||_{C^j_D}:= \sum_{|\alpha|+|\beta| \leq 
j}\left|\left|\left(\frac{\partial}{\partial z}\right)^\alpha 
\left(\frac{\partial}{\partial \bar{z}}\right)^\beta 
u\right|\right|_\infty
\end{equation}
 where $\alpha$ and $\beta$ are meant in the usual multi-index notation and 
 $|\alpha|$ and $|\beta|$ are their total orders. This norm defines the {\bf $\bm{C^j}$ topology} in the linear space of $C^j$ functions $u:D \to \C$ for which $||u||_{C^j(U)}$ is finite. Similarly we can define this norm for the spaces $C^\infty_{p,q}(D)$ and $C^\infty_p(D)$ by observing that the coefficients of the forms in these spaces are smooth complex valued functions. An alternative definition is given by taking the maximum over all $|\alpha|+|\beta| \leq j$ (instead of summing). When $D$ is relatively compact, this definition gives an equivalent norm.

Similarly, there is a distance function on the space $C^\infty_{0}(D)$ that defines the $C^\infty$-topology. This is defined for $f,g \in C^\infty_0(D)$ by 
\begin{equation}\label{eq:infty_norm}
\begin{split}
	d(f,g):= & \sum_{k=0}^\infty  \sum_{|\alpha|+|\beta| \leq k} \frac{1}{ k! (2n)! 2^{k+1}} \\ & \times \left|\left|\left(\frac{\partial}{\partial z}\right)^\alpha 
	\left(\frac{\partial}{\partial \bar{z}}\right)^\beta 
	(f-g)\right|\right|_\infty \left/  \left(1+ \left|\left|\left(\frac{\partial}{\partial z}\right)^\alpha 
	\left(\frac{\partial}{\partial \bar{z}}\right)^\beta 
	(f-g)\right|\right|_\infty \right)\right.
	\end{split}
\end{equation}

The previous norms and metric are defined  for
open sets in $\C^n$. The supremum norm defined in \cref{eq:sup_norm} is not coordinate dependent so its generalization to a complex manifold is exactly the same. On the other hand, the norm defined in \cref{eq:cj_norm} depends on coordinates but we can still generalize it to smooth complex valued functions on manifolds. Let $X$ be a compact manifold and fix  a finite collection of coordinate charts $\{(W_i,\psi_i)\}_{1 \leq i \leq \ell}$ that covers $X$. Let $u:X \to \C$ be a $C^j$ complex valued function, then we define:
\begin{equation}\label{eq:cj_norm_man}
||u||_{C^j_X}:= \max_{ 1\leq i \leq \ell}||u \circ \psi_i^{-1}||_{C^j_{W_i}}\end{equation}
In particular, this defines a norm in the space of smooth functions in compact analytic manifolds with boundary $\bar X$ which in turn induces a topology: the $C^j$ topology. Other equivalent norms appear in the literature: one may sum the norm over all the coordinate charts (instead of taking the max), or one may take  a partition of unity subordinated to the finite collection of coordinate charts and sum over all of them, by weighting  with the partition of unity functions as coefficients. Also, since $\bar X$ is compact, different choices of coordinate charts give equivalent norms.

This may be generalized to the space of sections of any vector bundle (see for example \cite[Definition 3.3]{Witt}) to define the usual $C^j$ topologies. In particular we can define the $C^j$ topology in the spaces $C^j_{p,q}(\bar
X)$ or $C^\infty_{p,q}(\bar{X})$.
 
 Similarly, we can define the {\bf $\bm{C^\infty}$ topology} on the spaces $C^\infty_{p,q}(\bar{X})$ and $C^\infty_{p}(\bar{X})$ by extending the definition of $d(f,g)$ to functions defined on a manifold via a fixed coordinate atlas. We also observe that if two functions $f,g:X \to \C$ satisfy that $||f-g||_{C^j_X}$ is small, it implies that their partial derivatives are close to each other up to the $j$-th order.

Finally, in order to make precise statements about perturbations of complex structures, we need to state which topology we are putting on the space of complex structures $\mathcal{J}(\bar{X})$ of an even-dimensional compact manifold with boundary. As we said in the beginning of the section, a  complex structure is an endomorphism $J:TX \to TX$ satisfying $J^2=-\id$ and an integrability condition. Thus, in particular it is a $(1,1)$ tensor. We simply take the $C^j$ and $C^\infty$ topologies on the space of $(1,1)$ tensors seen as the space of smooth sections of $T^*\bar{X} \otimes T\bar{X}$ (again we refer to \cite{Witt} for these definitions).  Consider the  induced topology on the subspace of complex structures. See \cite[page 34-35]{Greene} for a similar discussion in the more general setting of almost-complex structures.

\subsection*{Solutions to the $\bar\partial$-problem}
As we said in the introduction, the original work by Nirenberg and H\"ormander gave a solution to the $\bar{\partial}$-problem with estimates in the $L^2$-norm which are not enough for the purposes of this work.

Henkin in \cite{Hen}, using {\em integral methods}, gave a solution to the  $\bar{\partial}$-problem for $(0,1)$ forms in a bounded strongly pseudoconvex domain $D\subset \C^n$ with smooth boundary $\partial \bar D$. This is usually called the {\bf Henkin solution}. It has the following form:
\begin{equation}\label{eq:henkin_solution}
\begin{split}
u(z):=  \frac{(2n-1)!}{(2 \pi i)^n} &\left( \int_{\partial D \times [0,1]} \left< f \cdot d\bar\xi \right> \wedge \omega'(\eta) \wedge \omega(\xi) \right. \\
 & \left. -  \int_{D} \frac{\left< f \cdot (\bar{\xi}- \bar{z}) \right>}{\xi -z}  \omega(\bar{\xi}) \wedge \omega(\xi) \right)
\end{split}
\end{equation}
where
\begin{enumerate}
	\item $\omega'(\eta)= \sum_{k=1}^n (-1)^{k-1} \eta_k d\eta_1 \wedge \cdots \wedge d\eta_{k-1} \wedge d \eta_{k+1} \wedge \cdots \wedge d \eta_{n},$
	\item $\omega(\xi)=d\xi_1 \wedge \cdots \wedge d \xi_n,$
	\item $\left< \xi \cdot \eta \right> = \sum_{k=1}^n \xi_k \cdot \eta_k,$
	\item $\eta_k = \lambda \frac{\bar{\xi}_k - \bar{z}_k}{|\xi - z|^2} + (1-\lambda) \frac{p_k (z, \xi)}{\Phi(z, \xi)}$ for $k=1, \ldots, n$ with $\xi \in \partial D$ and $\lambda \in [0,1]$ (note that the parameter $\lambda$ does not explicitly appear in \cref{eq:henkin_solution} but $\eta$, which does appear, depends on $\eta_k$ which in turn depends on $\lambda$),
	\item $\Phi(z, \xi)=\sum_{i=1}^n(\xi_i-z_i)(p_i(z,\xi))$ (see \cite[d) in  p. 275]{Hen}) and satisfies certain other properties that are derived from \cite[Lemmas 2.4 and 2.7]{Hen2},
	\item $p_i(z, \xi)$ are certain well-defined functions which are holomorphic with respect to $z \in D_\delta:=\{z\in D: \rho(z)<\delta\}$ (where $\rho$ is an exhausting function for $D$ which is spsh near the boundary) for fixed $\xi$ and continuously differentiable with respect to $\xi \in \partial D$ for fixed $z$.
\end{enumerate}
	
   This function $u(z)$ (\cref{eq:henkin_solution}) satisfies that it is a smooth solution for the $\bar{\partial}$-problem $\bar{\partial}u = f$ and that there is a constant $K$ depending on $D$ such that $||u||_{\infty} \leq K ||f||_{\infty}$. Note that it is {\em canonical} because it is {\em explicitly} given by an integral. This is the main point that makes this result to be considered a breakthrough on the topic.

Later, in \cite{Leit}, Henkin and Leiterer proved global integral formulas that give an intrinsic solution for the more general case of Stein manifolds. However, this is not the approach to the manifold case that we follow in this work. See also \cite[Chapter 4]{Genn} for a compilation of results in this direction.

Before that, in \cite{Ker}, Kerzman proved the existence of solutions in the Stein manifold 
case with a uniform bound of the solution. In that work, first the result is proven  for 
{\em strongly pseudoconvex domains} (which he defines to be open sets in $\C^n$ which 
are Stein manifolds) \cite[Theorem 1.2.1]{Ker}. Then he notes that this result 
is as well valid for the case of strictly pseudoconvex domains {\em in} Stein 
manifolds \cite[Theorem 1.2.1'',   pg. 309]{Ker}.

\begin{thm}\label{thm:kerzman}
	Let $\bar X$ be a strongly pseudoconvex domain in a Stein manifold  and let $g  \in C^\infty_{0,1}(X)$ be a bounded $(0,1)$-form with $\bar \partial g = 0$. Then the equation $\bar{\partial} u = g$ has a bounded solution $u \in C^\infty_{0,0}(\bar X)$ with \begin{equation}
		||u||_\infty \leq K ||g(x)||_\infty
	\end{equation}
	where $K$ is a constant depending only on  $\bar X$.
\end{thm} 

Another interesting result that we need for the next subsection was achieved by Kohn. We refer to \cite{Foll} as the primer reference on this topic and \cite{Shaw} as a more recent book collecting those and more recent results. In particular, we are interested in the smooth {\bf Kohn solution} for strongly pseudoconvex manifolds (\cite[Corollary 5.3.11]{Shaw}):

\begin{thm}\label{thm:kohn_result}
	Let $\bar{X}$ be a strongly pseudoconvex manifold equipped with a Hermitian metric. Let $g \in C^\infty_{0,1}(\bar{X})$ with $\bar{\partial} g = 0$. Then there exists a unique $u \in C^\infty_{0,0}(\bar{X})$ such that $\bar{\partial}u=g$ and $u$ is orthogonal to $\ker(\bar{\partial})$. 
\end{thm}

\begin{remark}\label{rem:unique_kohn}
 In particular if $g=0$ then $u$ is $0$ (observe that any holomorphic function is {\em a} solution when $g=0$). We refer to \cite{Shaw} for more on this topic.
\end{remark}

\subsection*{Stability results for the $\bar{\partial}$-problem}
A natural generalization of 
these estimates consists of the determination of the regularity of the solutions to the 
$\bar{\partial}$-problem in terms of the regularity of a family of complex 
structures on a given strongly pseudoconvex space. In other words: if we consider a family of complex structures $J_t$ 
that varies smoothly with $t$ (all of them making $\bar{X}$ strongly pseudoconvex), in which topologies
do the solutions (with respect to those complex structures) move 
continuously?

This question is solved in different settings in \cite{Greene}. For instance they do it for the {\em Kohn solution} (in Section 3) and for the {\em Henkin solution} (in Section 4). In this subsection we state two of these stability results (one for each solution) and express them in a language useful for our purposes. 

We review first the stability results for the Henkin solution. As they explain in the introduction to Section 4 therein, {\em the Henkin solution
moves continuously in the $C^k$ topology provided that the complex structures vary sufficiently smoothly}.

 First,
\cite[Theorem 4.14, eq. 4.14.3']{Greene} proves a similar result for strongly 
pseudoconvex domains (relatively compact open sets with strongly pseudoconvex 
boundary) in 
$\C^n$. In the remark at the end of that same page it is explained that ``with some additional effort (see \cite{Ker}) the Henkin solution may
be locally transferred, via coordinate maps, to strongly pseudoconvex
manifolds and patched together''. 

  Then, in that same remark, Greene and Krantz explain that in this case, one must use the Sobolev embedding theorems to obtain estimates with respect to the $C^k$ norm rather than estimates in the Sobolev norms (which we don't use in this work). 
  
 A first inspection of that paper shows that the domain of the map of \cite[Theorem 4.14, eq. 4.14.3']{Greene} is \begin{equation}\label{eq:domain} \{\text{strongly pseudoconvex subdomains of $\hat{D}$} \}\times A_{0,1}(\hat{D})\end{equation} where $\hat{D}$ is an open neighborhood of a strongly pseudoconvex domain $D$ with smooth boundary and $A_{0,1}(\hat{D})$ is the space of closed $(0,1)$-forms. That is, the theorem is stated for deformations of  the strongly 
pseudoconvex domain within a given open neighborhood of a previously fixed strongly pseudoconvex domain. We need to define a  space different from that of  \cref{eq:domain} to make sense of the version of this result that we are going to use. We do this in the following remark.

\begin{remark}\label{rem:good_domain}
Eventually we are going to deform the complex structure. Hence, in our case, it does not make sense to use the space $A_{0,1}(\bar{X})$ of $\bar{\partial}$-closed $(0,1)$ forms, rather we need to consider the space of forms that are closed for {\em each} complex structure that we are considering on $\bar X$. Let $J_0$ be a complex structure that makes $\bar{X}$ strongly pseudoconvex. Since strong pseudoconvexity is an open property, there is a $C^\infty$ neighborhood of $J$ in the space of complex structures $\mathcal{J}(\bar{X})$ such that all the complex structures in this neighborhood make $\bar{X}$ strongly pseudoconvex.  Let $\mathcal{J}(\bar{X},J_0)$ be such a neighborhood. Consider the  product space $\mathcal{J}(\bar{X}, J_0) \times C^\infty_{1}(\bar{X})$ of the previous neighborhood and the space $C^\infty_{1}(\bar{X})$ of smooth differential $1$-forms with smooth complex valued functions as coefficients. And we define the subspace \begin{equation}
\begin{split}
\mathcal{JA}_{0,1}(\bar X,J_0):=& \{(J, g) \in \mathcal{J}(\bar{X}, J_0) \times C^\infty_{1}(\bar{X}):  \text{$g$ is a $(0,1)$-form with} \\ & \text{respect to $J$ and }\bar{\partial}_J g =0\}.\end{split}
\end{equation} 
This space has points that are pairs formed by a complex structure that makes $\bar{X}$ strongly pseudoconvex and a form that is both: a  $(0,1)$ form and is closed with respect to that complex structure. Hence, it makes sense trying to solve the problem $\bar{\partial}_J u =g$ for $(J,g)$ in this space. Whenever we put the product topology $C^k \times C^s$ on $\mathcal{J}(\bar{X}) \times C^\infty_{1}(\bar{X})$, it induces a topology on $\mathcal{JA}_{0,1}(\bar X, J_0)$  as a subspace. We call it the {\bf $\bm{C^{k,s}}$ topology}. 
\end{remark}

With the previous discussion in mind, we are ready to formulate a particular version of  \cite[Theorem 4.14]{Greene}:

\begin{thm}\label{thm:continuous_variation}
    Let $\bar X$ be a strongly pseudoconvex manifold with boundary with complex structure $J_0$. Then the map $$\mathcal{JA}_{0,1}(\bar X,J_0) \to C^1_{0}(\bar X)$$ that assigns to each $(J,g) \in \mathcal{JA}_{0,1}(\bar X)$ the patched Henkin solution $u=H_Jg$ to the problem $\bar{\partial}_J u = g$, is continuous from the $C^{j(k,n),k}$ topology to the $C^{k+ 1/2}$ topology. Here $j(k,n)$ is a function depending on $k$ and the dimension $n$ of $\bar{X}$ that increases with $k$ and $n$.
\end{thm}

The above result is not explicitly proven in \cite{Greene}, rather a weaker version (in strongly pseudoconvex domains in $\C^n$) is worked out and then a remark is made about the generalization. 

Since we could not find a reference for a proof of that theorem and we do not know how this {\em patching} of the Henkin solution is performed (presumably it uses \Cref{def:equivalence_strongly}),  we have decided to include another result (that appears in the same paper) but for which there is a proof. This result is similar in spirit but involves the {\em Kohn solution} instead of the Henkin solution.

Hence we turn our attention to the stability of the Kohn solution under perturbations of the complex structure. Here \cite[Theorem 3.10]{Greene} is the result of our concern. It gathers and extends results from \cite{Foll}. This theorem is stated in terms of Sobolev norms that we do not use in this work. However, the Sobolev embedding theorem (see for example in \cite[Proposition 2.3]{Greene}) relates these norms to the $C^j$ norms. Again, we observe that there is no hope in solving $\bar{\partial}u=f$ unless $f$ is $\bar{\partial}$-closed. Therefore, we use as domain of the function that assigns solutions, the same space described in \Cref{rem:good_domain}. 

Finally, for the convenience of the reader looking at \cite[Theorem 3.10]{Greene}, let us mention that in the notation used by Greene and Krantz we are taking $(p,q)=(0,1)$, $U$ is the whole manifold and $\xi=1$ on all $\bar{X}$. We are only stating the result directly for the $C^\infty$ topology because the family of $1$ forms to which we are going to apply the result varies smoothly with the parameter of the deformation.  As explained by Greene and Krantz, this result can be deduced after applying the Sobolev embedding theorem to obtain a $C^j$ stability for each $j$ and observing that two functions $f,g \in C^\infty(\bar{X})$ are  $C^\infty$ within $\epsilon$ topology once they are $C^K$ sufficiently close for some $K=K(\epsilon)$ (see \cite[p.50 before remark and pp.27-28]{Greene}). Also, recall that compactly supported smooth functions are contained in every Sobolev space and since the manifolds that we deal with are compact, every smooth function is compactly supported.

We have decided to state the strongest version of \cite[Theorem 3.10]{Greene} that we can use in our situation:

\begin{thm}\label{thm:kohn_stability}
Let $\bar{X}$ be a Hermitian strongly pseudoconvex manifold with complex structure $J_0$ and let $\mathcal{JA}_{0,1}(\bar X,J_0)$ be as in \Cref{rem:good_domain}. Then the map \[ \mathcal{JA}_{0,1}(\bar X,J_0) \to C^\infty_{0}(\bar X)\] that assigns to each complex structure $J$ and each $\bar{\partial}_J$-closed $(0,1)$ form $g$, the Kohn solution $u$ to the problem $\bar{\partial}_Ju=g$, is continuous from the $ C^{\infty,\infty}$ topology to the $C^\infty$ topology.
\end{thm}

In our case, we deal with {\bf families} of complex structures and $(0,1)$-forms. Let $Q$ be a manifold (in this work it will always be a disk) and let $t \in Q$ be a parameter in this space. A continuous map from $Q$ to  $\mathcal{J}(\bar{X},J_0), t \mapsto J_t$ equipped with the $C^\infty$ topology (recall \Cref{rem:good_domain}) is a {\bf smooth family of complex structures} over $Q$. We also say that the set $\{J_t\}_{t \in Q}$ is a smooth family of complex structures over $Q$. Similarly a continuous map from $Q$ to $C^\infty_{1}(\bar{X}), t \mapsto g_t$ equipped with the $C^\infty$ topology defines a {\bf smooth family} of differential $1$-forms $\{g_t\}_{t \in Q}$. In this language we can derive, from \Cref{thm:kohn_stability} the following:

\begin{thm}\label{thm:kohn_families}
	Let $\bar{X}$ be a Hermitian compact manifold with boundary and let $Q$ be a disk. Assume that $\{J_t\}_{t \in Q}$ is a smooth family of complex structures each of them making $\bar{X}$ strongly pseudoconvex and that $\{g_t\}_{t \in Q}$ is a smooth family of differential $1$-forms such that $g_t$ is a $(0,1)$-form with respect to $J_t$ and  $\bar\partial_{J_t}g_t=0$. Then the family $\{u_t\}_{t \in Q} \subset C^\infty_0(\bar{X})$ where each $u_t$ is the Kohn solution (\Cref{thm:kohn_result}) to the problem $\bar{\partial}_{J_t}u=g_t$ is a smooth family of complex valued functions over $Q$.
\end{thm}

For the proof of our crucial \Cref{prop:extension_to_family}, we could use either \Cref{thm:continuous_variation} or \Cref{thm:kohn_stability}. For reasons explained above about the lack of complete proofs of \Cref{thm:continuous_variation}, from now on, we only make references to the \Cref{thm:kohn_families} above,  as it depends only on \Cref{thm:kohn_stability}.

  We include a few other references that either use the above result, extend it or prove other very related results: \cite{Cho,Cho2,Cho3,Gong}.

\subsection*{Singularity theory and deformations of strongly pseudoconvex surfaces}

We introduce now a notion of {\em deformation} of a complex manifold and state the result on deformations of strongly pseudoconvex 
surfaces proven by Bogomolov and De Oliveira exploiting a previous result by Laufer.

\begin{definition}\label{def:deformation}
 Let $X$ be a strongly pseudoconvex manifold (maybe with smooth boundary), let $Q \subset \C$ be an open disk centred at $0\in \C$. We say that a map $\omega:\calX \to Q$ is a {\bf deformation} of $X$ if 
 \begin{enumerate}
 	\item $\calX$ is a complex manifold.
 	\item The map $\omega$ is a locally trivial $C^\infty$ fibration.
 	\item The central fiber $\omega^{-1}(0)$ is biholomorphic to $X$.
 \end{enumerate}
\end{definition}

\begin{remark}
	Another notion of deformation is as follows. Let $(X,x),(Y,y)$ and $(S,s)$ be germs of complex analytic spaces. A deformation of $(X,x)$ is a germ of a flat morphism $D:(Y,y) \to (S,s)$ together with an isomorphism between $D^{-1}(s)$ and $(X,x)$. 
\end{remark}

Suppose that we start with a strongly pseudoconvex manifold with smooth boundary $\bar X$. Then, it is known that if one wishes to keep track of the boundary of $X$, the theory of deformations of $X$ becomes infinite-dimensional \cite[Corollary 4.2]{Burn}. What Laufer did in his series of papers \cite{Lau2,Lau} to avoid this fact is to fix a compact analytic set $A \subset X$ and study deformations of $X$ near $A$. This allowed him to apply Kodaira-Spencer techniques to prove his results. It turns out that this setting is enough for many purposes, in particular for proving the existence of a versal deformation space \cite{Lau2}. 

The next result was obtained by Bogomolov and De Oliveira \cite[Theorem (2)]{Bog}.

\begin{thm}[Bogomolov-De Oliveira]\label{thm:small_stein_deformation}
	Let $\bar X$ be a strongly pseudoconvex surface with boundary that has a $1$-dimensional exceptional set without curves of the first kind: that is, smooth rational compact analytic curves with self intersection number equal to $-1$. Then there exists a deformation $\omega:\bar \calX \to Q$ over a small complex disk $Q\subset\C$ such that the fibers above $t \neq 0$ do not contain compact analytic curves without boundary.  
\end{thm}

 Laufer proved in \cite[Theorem 3.6]{Lau} a weaker result, allowing to get rid of any particular compact analytic closed curve without boundary $A$ with $A\cdot A < -1$ by slightly deforming the complex structure. 

The theorem above is another key ingredient for the proof of \Cref{thm:positive_factorization}.

\section{Preliminaries on mapping class groups}\label{sec:mapping_class_groups}

We turn now our attention to the theory of mapping class groups. The purpose of this section it to collect the results on mapping class group that we use throughout the rest of the text as well as  to fix notations and conventions.

\begin{notation}\label{not:preliminaries}
	The symbol $\Si_{g,r}$ denotes a compact connected surface of genus $g \geq 0$ and $r \geq 0$ boundary components. When it is clear from the context or unnecessary in the discussion, the subindex is dropped and we just write $\Si$. 
	
	We write $\modgr$ for the {\bf mapping class group} of the surface $\Si_{g,r}$, that is, the group of orientation-preserving diffeomorphisms of $\Si_{g,r}$ that restrict as the identity to $\partial \Si$ up to isotopy fixing the boundary pointwise. We write $\pmodgr$ for what is usually known as the {\bf pure mapping class group}, that is, the group of orientation-preserving diffeomorphisms of the surface $\Si_{g,r}$ that leave each boundary component invariant, up to isotopy free on the boundary. We also use the letter $\phi$ to denote a representative homeomorphism of the corresponding mapping class. It will always be clear from the context which is the case.
   
   Given a simple closed curve $\gamma \subset \Si$ we denote by $$t_\gamma: \Si \to \Si$$ a {\bf right-handed Dehn twist} around $\gamma$ or its mapping class in $\modgr$. The support of $t_\gamma$ is concentrated in a tubular neighborhood of $\gamma$. Also, the mapping class of $t_\gamma$ only depends on the isotopy class of the simple closed curve $\gamma$. 
   
   \begin{figure}[h]
   \includegraphics[scale=0.6]{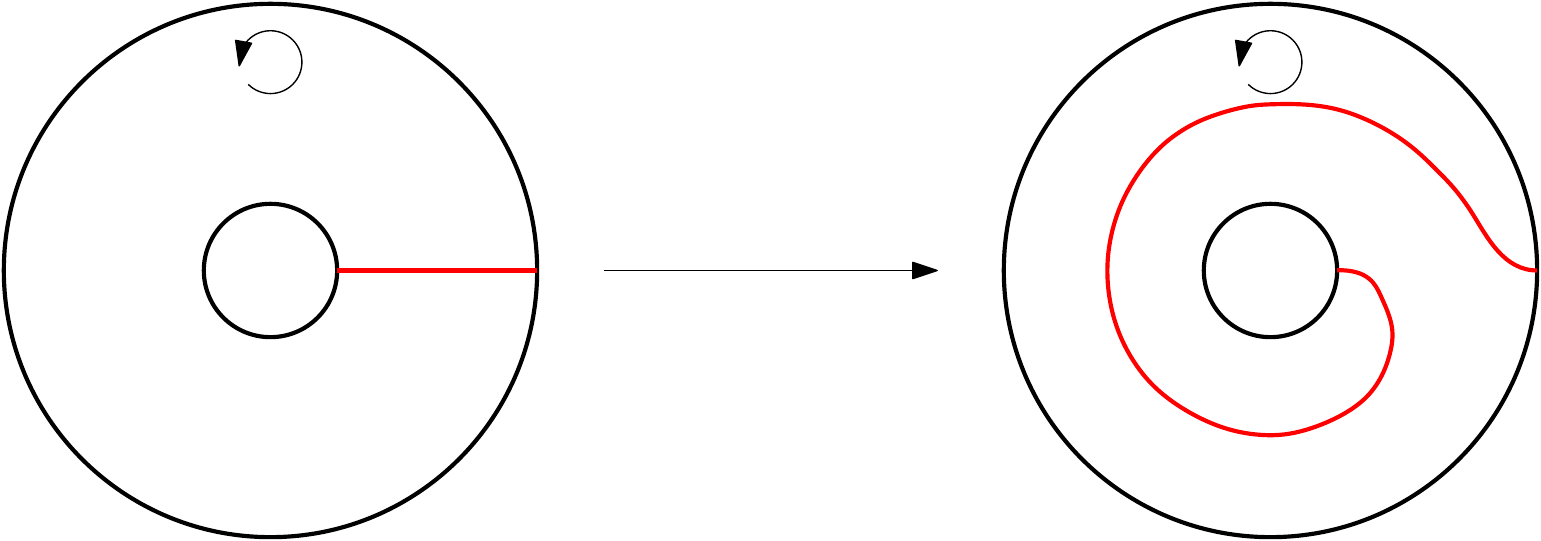}
   \caption{The effect of a right handed Dehn twist on a properly embedded arc in the annulus. The orientation of the annulus is indicated by the curved arrows.}
   \label{fig:r_dehn}
   \end{figure}

   Given a boundary component $B \subset \partial \Si$, we denote by $t_{B}$ the {\bf right-handed boundary Dehn twist} around $B$, which is by definition a right-handed Dehn twist around a simple closed curve parallel to $B$. The set of mapping classes that admit a factorization consisting only of right-handed Dehn twists is clearly a monoid and we denote it by $\mathbf{\dehn}$. We also call these factorizations {\bf positive}. 
\end{notation}
 
 The following is a landmark result in mapping class group theory which we state to motivate the next definition.
 
 \begin{thm}[See \cite{Thu} and Corollary 13.3 from \cite{Farb} ]\label{thm:nt_classification}
	Let $\phi: \Si \to \Si$ be an orientation preserving homeomorphism that 
	restricts to the identity on $\partial \Si$. Then there exists $\phi'$ isotopic to $\phi$ and a collection $\calC$ of non-nullhomotopic disjoint simple closed curves including all boundary components such that:
	\begin{enumerate}
		\item The collection of curves is invariant by $\phi'$, i.e. $\phi'(\calC)= \calC$.
		\item The homeomorphism $\phi'$ restricted to each connected component of 
		$\Si \setminus \calC$ (and its iterations by $\phi'$) is isotopic
		either to a periodic or to a pseudo-Anosov homeomorphism.
	\end{enumerate}
\end{thm}

The decomposition given by \Cref{thm:nt_classification} is called {\bf Nielsen-Thurston decomposition} and is unique up to isotopy if the system of curves $\calC$ is minimal by inclusion. We always assume that our pseudoperiodic representatives satisfy $(i)$ and $(ii)$ from the previous theorem. This decomposition leads to the following definition.

\begin{definition}\label{def:nt_decomposition}
	Let $\phi \in \modgr$, we say that $\phi$ is {\bf pseudoperiodic} if only periodic pieces appear in its Nielsen-Thurston decomposition.
\end{definition}

Pseudoperiodic homeomorphisms are of special importance in complex singularity theory, as all geometric monodromies of reduced holomorphic germs of functions on isolated complex surface singularities are of this kind. This is closely related to the fact that the links of isolated complex surface singularities are graph manifolds. To learn more on this topic we refer to \cite{Eis}. 

Actually, only a restricted type of pseudoperiodic homeomorphisms shows up in the situation mentioned in the previous paragraph. In order to describe precisely this class of homeomorphisms in \Cref{thm:anne_realization} below, we need to recall first the definitions of two numerical invariants associated to pseudoperiodic elements of a mapping class group.

\begin{definition}\label{def:fractional_dehn_twist}.
Let  $\phi \in \modgr$ be pseudoperiodic. Let $N$ be a tubular neighborhood of the union of all the boundary components and the curves in $\calC$. Let $B_i \subset \partial \Si_{g,r}$ be a connected component of the boundary and let $N_{B_i}$ be the connected component of $N$ which is a tubular neighborhood of $B_i$. By definition of pseudoperiodic homeomorphism, there exists $n \in \NN$  such that  $\phi^n$ is isotopic to a homeomorphism that restricts to the identity outside $N$ and  it is of the form $t_{B_i}^k$ when restricted to $N_{B_i}$ for some $k \in \ZZ$.  The {\bf fractional Dehn twist coefficient} of $\phi$ at $B_i$ is denoted by $\mathrm{fr}(\phi, B_i)$ and is defined to be the rational number  $$\mathrm{fr}(\phi, B_i):=\frac{k}{n}.$$ Note that $k$ might be negative meaning that $\phi^n$ is a power of a left-handed boundary Dehn twist near this boundary component.
\end{definition}

\begin{remark}
We have just defined the notion of fractional Dehn twist coefficient for pseudoperiodic mapping classes. However, there exists a notion of fractional Dehn twist coefficient for mapping classes that also have pseudo-Anosov pieces containing boundary components of the surface. We do not use this notion in the present work and we refer the interested reader to \cite[Section 3.2]{KoI}.
\end{remark}

 Let  $\phi \in \modgr$ be pseudoperiodic and let $C_1, \ldots, C_a$ be an orbit of curves in $\calC$ (as in \Cref{thm:nt_classification}) such that $\phi(C_i)=C_{i+1}$ for all $i=1, \ldots, a-1$ and $\phi(C_a)=C_1$. We say that the orbit is {\bf amphidrome} if $\phi^a_{|_{C_1}}:C_1 \to C_1$ inverts the orientation. Otherwise we say that it is a {\bf regular} orbit.
 
\begin{definition}\label{def:screw_number}
 Let  $\phi \in \modgr$ be pseudoperiodic and let $C_1, \ldots, C_a$ be an orbit of curves in $\calC$ (as in \Cref{thm:nt_classification}) such that $\phi(C_i)=C_{i+1}$ for all $i=1, \ldots, a-1$ and $\phi(C_a)=C_1$. Let $N$ be a tubular neighborhood of all the curves in $\calC$ and all the boundary curves. By definition of pseudoperiodic homeomorphism, there exists $n \in \NN$  such that  $\phi^n$ is isotopic to a homeomorphism that restricts to the identity outside $N$ and, in a tubular neighbourhood $N_{C_i} \subset N$ of any of the curves  $C_i$ in the orbit, is of the form $t^k_{C_i}$ for some $k\in \ZZ$. Let $\alpha=a$ if the orbit is regular and let $\alpha=2a$ if the orbit is amphidrome: that is, $\alpha$ is the smallest natural number such that $\phi^\alpha$ sends each curve in the orbit to itself preserving the orientation on the curve. We define the {\bf screw number} of $\phi$ at the orbit, as $$\mathrm{sc}(\phi,C_i):=\frac{k \alpha }{n}.$$
\end{definition}

	The screw number measures the amount of twisting of $\phi^\alpha$ around $C_i$ and it is the same for all curves in the same orbit. Therefore, it is well-defined either for the orbit itself or for any particular curve in $\calC$. 
	
	\begin{remark}
	 Observe that if we consider a boundary parallel curve as an orbit (consisting of just one curve) invariant by $\phi$ (recall \Cref{thm:nt_classification}) having a periodic piece of period $1$ on one of it sides, then the notions of screw number and fractional Dehn twist coefficient coincide. 
	 
	 	Our sign conventions may differ from those of other authors. For instance, they agree with those of Montesinos and Matsumoto \cite{Mat} for Dehn twist coefficients but they differ for screw numbers. Our conventions are chosen in order to make the invariant attached to a right handed Dehn twist positive, no matter if the curve is boundary parallel or any simple closed curve in $\calC$.
	\end{remark}

\begin{remark}\label{rem:formula_screw}
	We recall the following property of screw numbers. Let $\{C_1, \ldots, C_a\}$ be an orbit of invariant simple closed curves by $\phi$. Let $\beta=1$ if the orbit is regular and $\beta=2$ if the orbit is amphidrome and let $\alpha$ be as in \Cref{def:screw_number}. Take $C$ to be any curve in that orbit. Then $$\mathrm{sc}(t_{C}^m\circ\phi,C)=\frac{(k+ m \cdot n /a)\cdot \alpha}{n} = \mathrm{sc}(\phi,C)+ \beta \cdot m.$$ 
	That is, composing a pseudoperiodic homeomorphism with a power of a Dehn twist around any of the curves of a given orbit, varies the screw number associated with that orbit by the same (resp. twice ) amount  as the exponent of the power if the orbit is regular (resp. amphidrome). This follows from a direct application of the definition of screw number. 
	
	A similar phenomenon happens in the case of fractional Dehn twist coefficients:
	
	$$\mathrm{fr}(t_{B_i}^m\circ\phi,B_i)=\frac{(k+m \cdot n)}{n} = \mathrm{fr}(\phi,B_i)+m.$$ 
\end{remark}

The notion of fractional Dehn twist coefficient has been present in the literature for a long time and it is difficult to attribute it to a particular author. For example, Gabai \cite{Gab} already used these quantities to measure the difference between an homeomorphism admitting a pseudo-Anosov representative and this representative.

More recently, Honda, Kazez and Mati\'c \cite{KoI} used fractional Dehn twist coefficients as a topological tool to detect tight contact structures from the monodromy of a supporting open book. More concretely, they {\em recovered} the notion of a homeomorphism being {\bf right-veering} (see \cite[Definition 2.1]{KoI}) and defined the monoid  $\mathbf{\veer}$ consisting of all homeomorphisms of the surface $\Si_{g,r}$ that are right-veering. As they note in \cite{KoI}, this concept of ``veering to the right'' can be traced back to, at least, Thurston's proof of the left orderability of the braid group. See also \cite{Bogo2} as an example of a more recent use of this concept previous to \cite{KoI}.

 The following theorem contains just a couple of results of \cite{KoI} relevant for this work.

\begin{thm}\label{thm:veer_honda}
	Let $\phi \in \modgr$. Then
	\begin{enumerate}
		\item  $\dehn \subset \veer$ (\cite[Lemma 2.5]{KoI}).
		\item A contact structure on a $3$-manifold is tight if and only if all the supporting open books have right-veering geometric monodromy (\cite[Theorem 1.1]{KoI}).
	\end{enumerate}
\end{thm}

Also, we have the following theorem.

\begin{thm}[Loi-Piergallini \cite{Loi}, Giroux \cite{Gir}]\label{thm:dehn_fillable}
	A contact structure is holomorphically fillable if and only if there exists an open book supporting it with monodromy in $\dehn$.
\end{thm}

The above two theorems emphasize the importance for contact topology of being able to the determine the gap between $\veer$ and $\dehn$.

As we said in the introduction, this work is concerned with homeomorphisms which appear as monodromies of reduced holomorphic map germs defined on isolated complex surface singularities. Anne Pichon proved \cite[Th\'eor\`eme 5.4]{Pich}, using a deep result by Winters \cite[Theorem 4.3]{Win}, the following purely topological characterization of these homeomorphisms. This result establishes  the connection of \Cref{thm:positive_factorization} with the theory of mapping class groups (see \Cref{prop:equivalence}).

\begin{thm}[Pichon]\label{thm:anne_realization}
	Let $\phi: \Si \to \Si$ be an homeomorphism of a compact, connected and oriented surface with non-empty boundary. Then there exists an isolated complex surface singularity $X$ and a reduced holomorphic map germ $f: X \to \C$ whose associated monodromy is $\phi$ if and only if some power $\phi^n$ of the homeomorphism is isotopic to a composition of powers of right-handed Dehn twists around disjoint simple closed curves including curves parallel to all boundary components. This property is equivalent to all fractional Dehn twists and screw numbers being strictly positive.
\end{thm}

We name this important class of pseudoperiodic homeomorphisms once and for all:

\begin{definition}\label{def:fully_right}
	If $\phi \in \modgr$ is a pseudoperiodic homeomorphism with strictly positive fractional Dehn twist coefficients and screw numbers, we say that it is {\bf fully right-veering}.
\end{definition}

One can see in \cite{KoI} that the definition of right-veering homeomorphism refers only to the behavior of $\phi$ at the boundary of $\Si$. Roughly, a fully right-veering pseudoperiodic homeomorphism is a right-veering pseudoperiodic homeomorphism whose restrictions to all periodic pieces are also right-veering.

\section{Main theorem}\label{sec:main_thm}

In this section we prove that all fully right-veering pseudoperiodic homeomorphisms admit positive factorizations (\Cref{thm:positive_factorization}). We start by proving a proposition that has an interest of its own and is crucial in the proof this theorem.

\begin{prop}\label{prop:extension_to_family}
    Let $\omega:\bar\calX \to Q$ be a deformation of a strongly pseudoconvex 
    manifold  $\bar X:=\omega^{-1}(0)$ with smooth boundary, where $Q \subset  \C$ is a disk, and all the fibers are strongly pseudoconvex. Let $f:\bar X 
    \to 
    \C$ be a holomorphic map. Then there exists a neighborhood $Q' 
    \subset Q$ of $0$ and a smooth map $F:\omega^{-1}(Q') \to \C$ such that 
    $F|_{\omega^{-1}(t)}$ is holomorphic for all $t \in Q'$ and 
    $F|_{\omega^{-1}(0)}=f$.
\end{prop}

\begin{proof}
   We recall that by our \Cref{def:deformation}, $\omega:\bar \calX \to Q$ is actually a
   $C^\infty$ locally trivial fibration and, since $Q$ is contractible, it 
   is actually a trivial fibration. Therefore  there exists a diffeomorphism $ \bar X 
   \times Q \to \bar{\calX}$ which induces diffeomorphisms $T_t:\bar{X}_t \to 
   \bar{X}$ with $T_0 = \id$. This allows us to see the 
   deformation $\omega:\bar\calX \to Q$ as a family of complex structures $J_t$ 
   on $\bar{X}$ varying smoothly with $t \in Q$. For instance, if $J_t'$ is the 
   complex structure of $\bar{X}_t$, then one can consider $J_t := T_{t \ast} J_t'$ as the
   corresponding complex structure on $\bar{X}$. And by letting  $\bar{\partial}_t'$ be the complex conjugate derivative in $\bar{X}_t$ 
   and $\bar{\partial}_t$ be the corresponding one in $\bar X$ relative to the complex 
   structure $J_t$, we find that if the smooth functions $g,f:X \to \C$ satisfy  the equation 
   $$\bar{\partial}_t g = \bar{\partial}_t f$$ then, by precomposing, the equation 
   $$\bar{\partial}_t' (g\circ T_t) = \bar{\partial}_t' (f\circ T_t)$$ holds on $\bar{X}_t$ because $(\bar{X},J_t)$ and $(\bar{X}_t, J_t')$ are isomorphic as complex manifolds via $T_t$. This discussion tells us that we can solve 
   our problem in the central fiber of our deformation equipped with different 
   complex structures varying smoothly with $t$ and that this is equivalent to 
   solving the extension problem in each fiber of the deformation as the statement of
   the proposition requires.
   
  Take a possibly smaller $Q' \subseteq Q$ disk containing $0$ such that all the complex structures $J_t$ make $\bar{X}$ strongly pseudoconvex. This is always possible since strong pseudoconvexity is an open condition. By definition, $\bar{\partial}_t f$ is a 
   $\bar{\partial}_t$-closed $(0,1)$-form for all $t$. Since the complex structures vary smoothly, we
   find that  $\{\bar{\partial}_tf\}_{t \in Q'}$ is a smooth family of forms, each of them a $(0,1)$-form closed with respect to $\bar{\partial}_t$. Therefore, we can apply Greene and Krantz' result
   (\Cref{thm:kohn_families} derived from \Cref{thm:kohn_stability}) and get that there exists a smooth family of smooth
   solutions $\{u_t:\bar{X} \to \C\}_{t \in Q'}$ for the $\bar{\partial}_t$-problems 
   $$\bar{\partial}_t u = \bar{\partial}_t f.$$ 
   
  The solution in the central fiber is $0$ since $f$ is holomorphic and so $\bar{\partial}_0f=0$ (recall \Cref{rem:unique_kohn}). That is, we have found that $u_0 =  0$. Finally we define the sought $F:\omega^{-1}(Q')\to \C$ fiberwise by  
   $F|_{\bar{X}_t}:=f \circ T_t - u_t \circ T_t = (f-u_t)\circ T_t$. By 
   construction, $\bar{\partial}'_t (F|_{\bar{X}_t})=0$ (so each restriction is 
   holomorphic). Also by construction $F|_{\bar{X}_0}=f$ because $T_0=\id$ and $u_0=0$. And finally, $F$ is $C^\infty$  because $\{u_t\}_{t \in Q'}$ is a smooth family and because $f$ and $T_t$ are smooth.
\end{proof}

Before proving the main theorem of the paper, we state a proposition which establishes the equivalence between three classes of pseudoperiodic homeomorphisms. This is nothing else than a rewriting of Anne Pichon's result using the term {\em fully right-veering} introduced in \Cref{def:fully_right}:

\begin{prop}\label{prop:equivalence}
	The following classes of pseudoperiodic homeomorphisms coincide:
	\begin{enumerate}
		\item \label{it:i} Fully right-veering homeomorphisms.
		\item \label{it:ii} Geometric monodromies associated with reduced holomorphic functions defined on isolated complex surface singularities.
		\item \label{it:iii} Pseudoperiodic homeomorphisms for which there exists $n\in \NN$ such that $\phi^n$ is a composition of powers of right-handed Dehn twists around disjoint simple closed curves including all boundary components.
    \end{enumerate}  
\end{prop}

\begin{proof}
Class \cref{it:iii} is just the definition of \cref{it:i} (see \Cref{def:fully_right}). And that \cref{it:ii} is the same class of homeomorphisms as \cref{it:iii} is the content of  Anne Pichon's result (\Cref{thm:anne_realization}).
\end{proof}

We are ready to state and prove our main result.
\begin{thm}\label{thm:positive_factorization}
	All fully right-veering homeomorphisms admit a positive factorization.
\end{thm}

\begin{proof}

Let  $\phi:\Si \to \Si$ be a fully right-veering homeomorphism. By \Cref{prop:equivalence} (which is essentially Anne Pichon's result \Cref{thm:anne_realization}), there exists an isolated complex surface singularity $X$ and a reduced holomorphic function $f:X \to \C$ such that the monodromy of the corresponding Milnor-L\^e fibration on the tube \begin{equation}\label{eq:milnor_le_fibration}
f_{|f^{-1}(\partial D_\delta) \cap X\cap B_\epsilon}: f^{-1}(\partial D_\delta) \cap X \cap B_\epsilon \to \partial D_\delta
\end{equation} 
has fiber diffeomorphic to $\Si$ and monodromy conjugate to $\phi$ in $\modgr$.
	
	Let $\bar X$ be a Milnor representative of $X$, that is, $\bar X:= X
	\cap B_\epsilon$ where $B_\epsilon$ is a closed ball of small radius in a space where a representative of $X$ may be embedded centered at the singular point of $X$ and such that it
	intersects $X$ transversely and all the spheres contained in it centered at the singular point of $X$ also intersect $X$ transversely. We can assume as well that $\partial \bar{X}$ is strongly 
	pseudoconvex (recall \Cref{def:spc_stein}). The minimal resolution $\tilde{X} \to 
	X$ induces a map $\pi:\bar{\tilde{X}} \to \bar X$ where $\bar{\tilde{X}}$ is a strongly 
	pseudoconvex surface with smooth boundary that contains a nontrivial 
	$1$-dimensional exceptional set $A= \cup_j A_j$ such that no curve $A_j$ is exceptional of the first kind (smooth and rational compact analytic curve with $A\cdot A=-1$). By Laufer and Bogomolov-De Oliveira's result 
	(\Cref{thm:small_stein_deformation})  there exists a deformation 
	$\omega:\bar \calX \to Q$ of $\bar{\tilde{X}}=\omega^{-1}(0)$ over a small disk such 
	that no fiber $\bar X_t:=\omega^{-1}(t)$ with $t \neq 0$ contains compact 
	analytic curves without boundary. By taking the deformation disk $Q$ small enough, we can ensure that the 
	boundaries $\partial \bar X_t$ are strongly pseudoconvex. Indeed, the 
	central fiber has a strongly pseudoconvex boundary and this is a property 
	which is stable under small perturbations. Hence, we can assume that $\bar 
	X_t$ is a Stein domain for all $t \neq 0$.
	
Since the morphism $\pi:\bar{\tilde{X}} \to \bar X$ induced by the minimal resolution is an isomorphism outside the singular point of $X$, the Milnor-L\^e fibration of \cref{eq:milnor_le_fibration} lifts to an equivalent fibration \begin{equation}\label{eq:milnor_le_res}
\tilde{f}_{|\tilde{f}^{-1}(\partial D_\delta) \cap \bar{\tilde{X}} }:\tilde{f}^{-1}(\partial D_\delta) \cap \bar{\tilde{X}}  \to \partial D_\delta,
\end{equation}
where $\tilde{f}:=f\circ\pi.$

Now we apply \Cref{prop:extension_to_family} to our deformation $\omega$ and we get a, maybe, smaller disk $Q' \subset Q$ and a map $F: \omega^{-1}(Q') \to Q'$ that extends $\tilde{f}|_{\bar{\tilde{X}}}:\bar{\tilde{X}} \to \C$. Moreover, that same proposition tells us that the restriction of $F$ to $\bar{X}_t$ is holomorphic for all $t \in Q'$ and the first order partial derivatives of $F$ vary continuously with $t$ because $F$ is $C^\infty$.
	
	Let $\tilde{f}_t:=F|_{\bar X_t}$. Since transversality is an open condition and the partial derivatives of the functions $\tilde{f}_t$ vary continuously on $t$ we find that there exists a (maybe smaller) disk $Q'' \subset Q'$ containing $0$ such that the maps
	
	\begin{equation}\label{eq:milnor_t_fibration} f_{t|\tilde{f}_t^{-1}(\partial D_\delta) \cap \bar X_t}:\tilde{f}_t^{-1}(\partial D_\delta) \cap \bar X_t \to \partial D_\delta\end{equation}
	are locally trivial fibrations equivalent to \cref{eq:milnor_le_res} (and so to \cref{eq:milnor_le_fibration}) for all $t \in Q''$.
	
 Again, by shrinking maybe a little bit more $Q''$, we can assert  that $\tilde{f}_t$ 
 has no critical points on $\partial \bar X_t$ because $f$ did not have 
 critical points on $\partial \bar{\tilde{X}}$ and being a submersion is a stable 
 property. Since $\bar X_t$ is a Stein domain, it cannot contain compact analytic curves without boundary. Moreover, the set of critical points of $$f_{t|\tilde{f}_t^{-1}(D_\delta) \cap \bar X_t}:\tilde{f}_t^{-1}(D_\delta) \cap \bar X_t \to D_\delta,$$ is a compact analytic set that does not intersect the boundary (because $f$ is a submersion restricted to $\partial \bar{X}_t$). We conclude that this critical set must consist of a finite collection of (isolated) points.
 
 Therefore, the germs of the fibers of $f_t$ at the critical points of $f_t$ are isolated plane curve singularities. We conclude by a direct application of Picard-Lefschetz theory 
 that the monodromy of \cref{eq:milnor_t_fibration} admits a positive factorization since it is a composition of monodromies of isolated plane curve singularities (and these are known to admit morsifications). Then the monodromy of the equivalent fibrations \cref{eq:milnor_le_res} and \cref{eq:milnor_le_fibration} also admit positive factorizations.
\end{proof}

In \cite{KoII}, Honda, Kazez and Mati\'c  proved that every right-veering and freely periodic homeomorphism of the punctured torus admits a positive factorization. This proof is carried out by checking each case since the list of freely periodic homeomorphisms of the torus is short. Next we observe that the previous theorem is a full generalization of this fact, in particular we obtain the following straightforward corollary.

\begin{cor}\label{cor:periodic}
	Freely periodic homeomorphisms with positive fractional Dehn twist coefficients admit positive factorizations.
\end{cor}

Next, we prove a second corollary of \Cref{thm:positive_factorization}. It is a generalization of \cite[Theorem 4.2]{Col}
where the authors prove it for freely periodic monodromies with positive fractional Dehn twist coefficients.

\begin{cor}\label{cor:colin_honda}
	Let $\phi:\Si \to \Si$ be a fully right-veering homeomorphism, then the associated contact structure is Stein fillable.
\end{cor}

We give two proofs of this corollary. The first one is a direct consequence of our main result \Cref{thm:positive_factorization} and a result by Giroux \cite{Gir} and Loi-Piergallini \cite{Loi}. The second one does not use our result and relies directly on the result by Anne Pichon and the result by Laufer-Bogomolov-DeOliveira.

\begin{proof}[First proof of \Cref{cor:colin_honda}]
By \cite[Theorem 3.4]{Loi}, \cite[Corollaire 7]{Gir}, if a monodromy admits a positive factorization, then its associated contact structure is Stein fillable.
\end{proof}

\begin{proof}[Second proof of \Cref{cor:colin_honda}]
	By Anne Pichon's result, \Cref{thm:anne_realization}, we know that the open book associated with $\phi$ appears as the link of an isolated complex surface singularity. Take a resolution of the singularity and apply Laufer-Bogomolov-DeOliveira's result, \Cref{thm:small_stein_deformation}, to obtain a Stein surface with isomorphic strongly pseudoconvex boundary which realizes the Stein filling. Note that by a construction made by Patrick Popescu in \cite[Proposition 5.8]{Pop}, this surface is a Milnor fiber of a complex surface singularity (in general non-normal) with the given topological type.	
\end{proof}

\begin{remark}\label{rem:stein_fillability}
	It was believed for some time that if a contact structure is Stein fillable then {\em all} of its associated open
	books have monodromy admitting a positive factorization. This was proven
	false by giving quite explicit examples in at least two different papers: \cite[Theorem 1.3]{Wand} and \cite[Theorem 1.1]{Bak}.
	
	In particular, we remark that there was no known way of going from  \cite[Theorem 4.2]{Col} to our \Cref{cor:periodic}. 
\end{remark}

We pose the following, weaker question,  that seems reasonable in view of the results of the present paper:

\textbf{Question.} Let $\phi \in \modgr$ and assume that the associated contact structure is Milnor fillable in the sense of \cite{Cau}, that is, it is contactomorphic to the link of some
isolated complex surface singularity. Is it then true that $\phi$ admits always
a positive factorization? What if we impose that $\phi$ be pseudoperiodic? 

\section{Consequences of the main theorem}\label{sec:consequences}

In this section we define an invariant of a mapping class $\phi \in \modgr$ for all $g,r$ except for a small family of degenerate cases. This invariant is a non-negative integer that measures how far is a mapping class from being in $\dehn$. This together with a theorem from \cite{Horn} and our main result (\Cref{thm:positive_factorization}) gives sufficient conditions for a pseudoperiodic homeomorphism to be in $\dehn$ (see \Cref{thm:criterion}).

First we introduce a notion that we use in this section. For a surface $\Si_{g,r}$ denote by $\calB$ the mapping class resulting of the composition of a single right-handed Dehn twist around each boundary component, that is, $\calB := t_{B_1} \cdots t_{B_r}$. We call $\calB$ a {\bf boundary multitwist}. 

For all $g \geq 0$ and $r=0$  all the elements in the pure mapping class group $\pmodgr$ (recall \Cref{not:preliminaries}) admit a positive factorization along non-separating simple closed curves. A proof of this statement can be found in \cite[page 124]{Farb} under the name of ``A strange fact''. In the \Cref{prop:pos_fact} we observe that a similar argument together with the main result from  (\cite{Horn}) implies a much more general result. Observe that throughout this section we drag the constraint of having to deal with non-separating simple closed curves. We explain the impact of this constraint in \Cref{rem:cons_non_separating}.

We recall the piece of \cite[Theorems A and B]{Horn} that we use in this work:

\begin{thm}\label{thm:positive_fact_boundary}
Given $\phi \in \modgr $, let $L(\phi)$ be the supremum taken over the number of positive Dehn twists  around non-separating simple closed curves appearing in a positive factorization of $\phi$ . Then

\[
L(\calB) =
\begin{cases}
+ \infty & \text{if $g=1$ and $r >9$ or $g \geq 2$ and $r >4g+4$}  \\
- \infty  & \text{if $r \leq 2g -4$} \\
\mathrm{finite} &  \mathrm{otherwise}
\end{cases}
\]

and 

\[
L(\calB^k) =
\begin{cases}
12k & \text{if $g=1, k \geq 1$ and $r \leq 9$}  \\
+ \infty  & \text{if $g \geq 2$ and $k \geq 2$} \\
\end{cases}
\]
The symbol $+\infty$ means that arbitrarily long factorizations exists, and $-\infty$ means that there does not exist such factorizations.
\end{thm}

\begin{remark}[Genus $0$ case]\label{rem:case_g0}
	Observe that in the theorem above, as stated in \cite{Horn},  the $g=0$ case is included in the ``otherwise'' part of the of the definition of $L(\calB)$, and therefore $L(\calB)$ evaluates as ``finite''. This cannot be the case since when $g=0$ there are no curves that are at the same time non-nullhomotopic and non-separating  (cutting along an essential non-separating simple closed curve reduces genus by $1$). Hence the case of genus zero is not covered in the theorem. We use this theorem heavily in this section, so we leave out this case as well.
	
	However, we  check that this does not leave out a lot of pseudoperiodic homeomorphisms: the class of pseudoperiodic elements of $\mathrm{Mod}_{0,r}$ is not very rich. 
	
	When $r=1$ we are dealing with a disk so there is not much to do here. If $r=2$ the surface is an annulus, $\mathrm{Mod}_{0,r} \simeq \ZZ$ and it is trivial to know which elements admit a positive factorization.
	
	Suppose now that $r\geq 3$. Let us look first for the freely periodic elements of  $\mathrm{Mod}_{0,r}$, this is the same as looking for the periodic elements of $\mathrm{PMod}_{0,r}$. Since these periodic elements admit periodic representatives (by the Nielsen realization theorem \cite[Theorem 7.2]{Farb}), this is equivalent to looking for periodic homeomorphisms of the sphere with $r$ fixed points. But a direct application of Hurwitz formula shows that these do not exist. 
	
	We claim that all the orbits in $\calC$ (recall \Cref{thm:nt_classification}) consist of just one curve. Let $C_1, \ldots, C_k$ be an orbit. All these curves must be separating. Since they are non-nullhomotopic it must be the case that the connected components on  both sides of each curve contain at least one boundary component of $\Si_{0,r}$. Since we are considering homeomorphisms  that fix boundary components these components must be invariant by $\phi$ and so each of the curves $C_1, \ldots, C_k$ must be invariant, therefore $k=1$.
		
	Putting together the previous two paragraphs we obtain that any pseudoperiodic homeomorphism of $\Si_{0,r}$ consists of a composition of powers of Dehn twists around disjoint simple closed curves.
	
	Having described the pseudoperiodic homeomorphisms of $\Si_{0,r}$ and in view of the fact that there is no hope to apply \Cref{thm:positive_fact_boundary} in this context, we do not deal with the case $g=0$ in the rest of the section.
\end{remark}

We go back to what we announced at the beginning of this section and prove that $\pmodgr$ is generated as a semigroup by right-handed Dehn twists around non-separating simple closed curves in a lot of  cases. For this first result we actually only use the second part of \Cref{thm:positive_fact_boundary} which corresponds with a portion of \cite[Theorem B]{Horn}.

\begin{prop}[\textbf{A stranger fact}]\label{prop:pos_fact}
	Every element in the pure mapping class group $\pmodgr$ can be written as a composition of right-handed Dehn twists along non-separating simple closed curves for all $g \geq 1$ and $r \geq 1 $ except for the cases when one has simultaneously $g=1$ and $r > 9$.
\end{prop}

\begin{proof}
	We show that one can use the same trick used by Farb and Margalit in \cite[page 124]{Farb}. Observe that the second part of \Cref{thm:positive_fact_boundary} implies that for all $g \geq 1, r \geq 1$ (except when $g=1$ and $r > 9$ at the same time), there exists in $\modgr$ a positive factorization of a power of a boundary multitwist $\calB^k$ around non-separating simple closed curves.
	
	Consider the natural surjective homomorphism of groups $\modgr \to \pmodgr$.
	For $\phi\in \modgr$ we denote by $\hat{\phi}$ its image in $\pmodgr$ by the previous projection.
	
	A boundary multitwist $\calB$ is the identity when seen in $\pmodgr$, that is, $\hat{\calB}=\hat{\id}$. Therefore we find that the identity $\hat{\id} \in \pmodgr$ can be expressed as a (non-empty) product of right-handed Dehn twists along non-separating simple closed curves $\gamma_1, \ldots, \gamma_k$. That is
	$$\htt_{\gamma_k} \cdots \htt_{\gamma_1} = \hat\id.$$
	Therefore, multiplying both sides to the right by $\htt_{\gamma_1}^{-1}$ yields a positive factorization around non-separating simple closed curves of a left-handed Dehn twist around a non-separating simple closed curve. This together with the Change of Coordinates Principle \cite[pg 37]{Farb} which says that all non-separating simple closed curves are the same up to conjugation and the fact that $\pmodgr$ is generated by (right and left-handed) Dehn twists along non-separating simple closed curves gives the result. Just observe that we can substitute any left handed Dehn twist in a factorization in $\pmodgr$ by a composition of right-handed Dehn twists.
\end{proof}

As a consequence  we obtain the following corollary, which shows that composing with enough positive boundary Dehn twists stabilizes the whole mapping class group $\modgr$ into $\dehn$. 

\begin{cor}\label{cor:obstruc_dehn}
	Let $\phi \in \modgr$ and let $t_{B_{1}}, \ldots, t_{B_r}$ be right-handed boundary Dehn twists. Then for $n_1, \ldots,  n_r \in \ZZ_{\geq0}$ big enough the element
	$$t_{B_{r}}^{n_r} \cdots t_{B_{1}}^{n_1}\phi$$
	is in $\dehn$ except in the cases $g=1$ and $r > 9$.
\end{cor}

\begin{proof}
	Suppose that $\phi$ cannot be written as a composition of right-handed Dehn twists. Take its image $\hat\phi$ in $\pmodgr$. By \Cref{prop:pos_fact}, $\hat\phi$ can be expressed as a product of right-handed Dehn twists in $\pmodgr$.
	
	Take the composition of these right-handed Dehn twists in $\modgr$ and denote the resulting mapping class by $\tilde{\phi}$. Observe that $\phi \circ \tilde{\phi}^{-1}$ is a composition of powers of boundary Dehn twists and some of them are left-handed Dehn twists. Just compose $\phi$ with the inverse of these powers.
\end{proof}

\begin{remark}
It is worth noting that it is a trivial observation that derives from the definition of right-veering homeomorphism in \cite{KoI} that composing with enough right-handed boundary Dehn twists, one can get any homeomorphism to be
right-veering. The above corollary says that, except in the {\em exceptional cases}, composing with enough right-handed boundary Dehn twists puts any monodromy in $\dehn$!
\end{remark}

 The previous corollary contains the idea that there is an obstruction localized at the fractional Dehn twist coefficients of $\phi$ for an homeomorphism to be in $\dehn$ . Therefore, by means of \Cref{cor:obstruc_dehn} the following is well-defined:
 
\begin{definition}\label{def:correcting}
	Let $\Si_{g,r}$ be a surface such that $g > 1$ or with $g=1$ and $r\leq 9$. Let $\phi \in \modgr$. We define $N_\phi$ as the minimal natural number such that $$\calB^{N_\phi}\phi \in \dehn.$$ We call $N_\phi$ the {\bf correcting exponent} of $\phi$.  
\end{definition}

\begin{definition}\label{def:essential_element}
	We say that a pseudoperiodic mapping class $\phi \in \modgr$ is an {\bf essential} mapping class in $\modgr$ if $|\fr(\phi,B_i)|<1$ for all boundary components, $|\mathrm{sc}(\phi, C_j)| < 1$ for all regular orbits of invariant curves in $\calC$ and $|\mathrm{sc}(\phi, C_j)| < 2$ for all amphidrome orbits.
\end{definition}

Let $x \in \QQ$. Denote $\mathrm{int}(x)\in \ZZ$ the following variant of its integer part: 
$$
\mathrm{int}(x) =
\begin{cases}
\floor*{x} & \text{if $x \geq 0$}  \\
\ceil*{x} & \text{if $x < 0$} \\
\end{cases}
$$

\begin{lemma}\label{lem:essential_part}
	Let $\Si_{g,r}$ be a surface with $g,r > 0$ and let $\phi \in \modgr$ be pseudoperiodic. Assume that $\phi$ has exactly $s$ different orbits of invariant curves in its Nielsen-Thurston decomposition and let $C_1, \ldots,C_s$ be a sequence of representatives of those orbits. Then there exist  unique integers $n_1, \ldots, n_r, m_1, \ldots, m_s$ such that
	\begin{enumerate}
		\item \label{it:es_i}The mapping class
		$\tilde{\phi}:=\phi \cdot t_{B_{1}}^{n_1}\cdots
		t_{B_{r}}^{n_r} t_{C_{1}}^{m_1}\cdots t_{C_{s}}^{m_s}$ is essential.
		\item \label{it:es_ii}If $\fr(\tilde\phi,B_i) \neq 0$, its sign is the same as $\fr(\phi,B_i)$.
		\item \label{it:es_iii} If $\mathrm{sc}(\tilde \phi,C_i) \neq 0$ its sign is the same as $\mathrm{sc}(\phi,C_i)$.
	\end{enumerate} 
\end{lemma}

\begin{proof}
	 We define $n_i:=-\mathrm{int}(\fr(\phi,B_i))$ for $i=1, \ldots, r$ and $m_j:= -\mathrm{int}(\mathrm{sc}(\phi,C_j)/\beta_j)$ for $j=1, \ldots, s$ where $\beta_j=1$ if the corresponding orbit is regular and $\beta_j=2$ if the corresponding orbit is amphidrome.
	 
	  Then by applying the formulas in  \Cref{rem:formula_screw} to each orbit and each boundary component  we obtain $$\mathrm{sc}(\tilde \phi,C_i)=\mathrm{sc}(t_{C_i}^{m_i}\phi,C_i) =\mathrm{sc}(\phi,C_i)-\mathrm{int}(\mathrm{sc}( \phi,C_i))$$ for regular orbits and $$\mathrm{sc}(\tilde \phi,C_i)=\mathrm{sc}(\phi,C_i)-\mathrm{int}(\mathrm{sc}( \phi,C_i)/2)2$$ for amphidrome orbits. An analysis of the last term in the last equations yields
	  
	  \[
	 \mathrm{int}(\mathrm{sc}( \phi,C_i)/2)2 =
	  \begin{cases}
	  \mathrm{int}(\mathrm{sc}( \phi,C_i))-1 & \text{if $\mathrm{int}(\mathrm{sc}( \phi,C_i))>0$ and odd}  \\
	  \mathrm{int}(\mathrm{sc}( \phi,C_i))+1 & \text{if $\mathrm{int}(\mathrm{sc}( \phi,C_i))<0$ and odd}  \\
	  \mathrm{int}(\mathrm{sc}( \phi,C_i)) & \text{if $\mathrm{int}(\mathrm{sc}( \phi,C_i))$ is even} \\
	  \end{cases}
	  \]
	  
	  Similarly we have  $$\fr(\tilde{\phi},B_i)=\fr(t_{B_i}^{n_i}\phi,B_i)=\fr(\phi,B_i)-\mathrm{int}(\fr(\phi,B_i))$$ for each boundary component. 
	  
	  By the previous formulas we find that the signs of fractional Dehn twists coefficients and screw numbers of $\phi$ and $\tilde{\phi}$ are preserved whenever their corresponding invariants are non-zero. Moreover, we also have that $|\fr(\tilde{\phi},B_i)| < 1$, that $|\mathrm{sc}(\tilde{\phi},C_i)| < 1$ for a regular orbit and that $|\mathrm{sc}(\tilde{\phi},C_i)| < 2$ for an amphidrome orbit. Actually, we can refine the last inequality a little bit more and  say that for amphidrome orbits: $|\mathrm{sc}(\tilde{\phi},C_i)| < 1$ if $\mathrm{sc}(\phi,C_i)$ is even and $1<|\mathrm{sc}(\tilde\phi,C_i)| < 2$ if $\mathrm{sc}(\phi,C_i)$ is odd.

	   We conclude that \cref{it:es_i}, \cref{it:es_ii} and \cref{it:es_iii} from the statement of the lemma are satisfied.
	
	 Now we prove the uniqueness part of the statement. Take some $n_i' \in \ZZ$ different from the selected $n_i$. We are going to see that if instead of considering $\tilde{\phi}$, we consider the mapping class  
	$$\tilde{\phi}':=\phi \cdot t_{B_{1}}^{n_1}\cdots t_{B_{i}}^{n_i'} \cdots
	t_{B_{r}}^{n_r} t_{C_{1}}^{m_1} \cdots t_{C_{s}}^{m_s},$$ where we have substituted $n_i$ by $n_i'$, then either \cref{it:es_i} or \cref{it:es_ii} is violated. Observe that by the second formula in \Cref{rem:formula_screw}, $$\fr(\tilde{\phi}',B_i)= \fr(\tilde{\phi},B_i) + n_i'-n_i.$$
	Since $n_i$ and $n_i'$ are distinct  integers, we have that $|n_i'-n_i| \geq 1$. On another hand we see that, because $\tilde{\phi}$ is essential, $|\fr(\tilde{\phi},B_i)|<1$. Putting these two last facts together we find that  either $$|\fr(\tilde{\phi},B_i) + n_i'-n_i| \geq 1$$ or  $$\mathrm{sign}(\fr(\tilde{\phi},B_i) + n_i'-n_i)=-\mathrm{sign}(\fr(\tilde{\phi},B_i)) =-\mathrm{sign}(\fr(\phi,B_i))$$
	which means that either $\tilde{\phi}'$ is not essential or the sign of  its fractional Dehn twist coefficient at $B_i$ is different from that of $\phi$.
	
	 A similar reasoning yields that choosing some $m_i' \neq m_i$ makes \cref{it:es_i} or  \cref{it:es_iii} false.
\end{proof}

After the previous lemma, the following is a natural definition.

\begin{definition}\label{def:essential_part}
Let $\phi \in \modgr$ be a pseudoperiodic mapping class. Its {\bf essential part} is the conjugacy class of the mapping class of $\tilde{\phi}$ characterized in \Cref{lem:essential_part}.
\end{definition}

\begin{remark}
	The screw numbers and fractional Dehn twist coefficients are conjugacy and isotopy invariants. Also recall that for any simple closed curve $\gamma$ and any homeomorphism $\psi$ we have that the equality $\psi \circ t_\gamma \circ \psi^{-1}= t_{\psi(\gamma)}$ holds in $\modgr$. Since the statement of \Cref{lem:essential_part} does not impose conditions on the curves that we are selecting for each orbit, it only makes sense to define the essential part associated to a given pseudoperiodic homeomorphism as a conjugacy class.
\end{remark}

Observe that $\phi$ is essential if and only if $\phi$ is in the conjugacy class of $\tilde{\phi}$.
Putting together the results of this section yields the following sufficiency criterion for a general family of pseudoperiodic mapping classes to be in $\dehn$.

\begin{thm}\label{thm:criterion}
	Let $\phi \in \modgr$ be a pseudoperiodic homeomorphism for some $g,r\geq 1$ except for the cases when simultaneously $g=1$ and $r>9$. Assume $\fr(\phi,B_i) >0$ for all $i=1, \ldots,r$. Let us denote by $\sn_1, \ldots, \sn_b < 0$ its negative screw numbers and assume that the corresponding orbits consist of non-separating simple closed curves. For each of the $j=1, \ldots, b$ orbits, let $\beta_j=1$ if the orbit is regular and $\beta_j=2$ if the orbit is amphidrome. Let
	\[
	k =
	\begin{cases}
	1 & \text{if $g=1, r <9$ or $g \geq 2,r \leq 2g - 4$}  \\
	2 & \text{if  $g \geq 2, r > 2g - 4$.}\\
	\end{cases}
	\]
	For each $j=1,\ldots,b$, let $d_j:=  -\mathrm{int}(\sn_j/\beta_j)+1.$ If \begin{equation}\label{eq:sufficient}
		k\sum_{j=1}^b d_j  \leq \min_{i\in\{1, \ldots, r\}} \fr(\phi, B_i)\end{equation} then $\phi$ admits a positive factorization.
\end{thm}

\begin{proof}
	Recall (see \Cref{rem:formula_screw}) that composing $\phi$ with a right-handed Dehn twist around any curve of an orbit of $\mathcal{C}$ increases the screw number of that orbit by $1$ if the orbit is regular and by $2$ if the orbit is amphidrome.
	
	It follows from \Cref{thm:positive_fact_boundary}  that given a non-separating simple closed curve $\gamma$, there exist a collection of non-separating simple closed curves $\alpha_1, \ldots, \alpha_c$ such that  \begin{equation}\label{eq:pos_fact_boun}
	\calB^k t_{\gamma}^{-1}= t_{\alpha_1} \cdots  t_{\alpha_c}\end{equation} holds in $\modgr$ where $k$ is like in the statement of the present theorem. That is, we can say that a left-handed Dehn twist, {\em costs} a boundary multitwist if $g=1, r <9$ or $g \geq 2,r \leq 2g - 4$ and that it {\em costs} $2$ boundary multitwists if $g \geq 2, r \leq 2g - 4$.
	
	Boundary multitwists commute with every element in $\modgr$ because their support is disjoint from any non-separating simple closed curve. Let $\gamma$ be one of the representatives  of the orbits considered in the statement of the theorem (those that have negative screw numbers and consist of non-separating simple closed curves). Consider $\calB^{-k}t_\gamma\phi$ (where $k$ is as in the hypothesis). Then, invoking the formulas from \Cref{rem:formula_screw} we find
	\begin{enumerate}
		\item that $\fr(\calB^{-k}t_\gamma\phi, B_i) = \fr(\phi, B_i)-k$ for all boundary components $B_i$ and
		\item that $\mathrm{sc}(\calB^{-k}t_\gamma\phi, \gamma) = \mathrm{sc}(\phi, \gamma)+ \beta_j.$ 
	\end{enumerate}
	By the same reasoning, changing $\calB^{-k} t_\gamma$ for $\calB^{-kd_j}t_\gamma^{d_j}$ we get	
	
		\begin{enumerate}[label= \arabic*)]
		\item that $\fr(\calB^{-kd_j}t_\gamma^{d_j}\phi, B_i) = \fr(\phi, B_i)-d_j k$ for all boundary components $B_i$ and
		\item \label{it:screw_in}that $\mathrm{sc}(\calB^{-kd_j}t_\gamma^{d_j}\phi, \gamma) = \mathrm{sc}(\phi, \gamma)+ \beta_j d_j$ where $\beta_j=1$ or $2$ depending whether the orbit in which $\gamma$ lies is regular or amphidrome.
	\end{enumerate} 

	If we do this for a curve $\gamma_1, \ldots, \gamma_b$ in each orbit we get a homeomorphism $$\hat{\phi}:= \left(\calB^{-k \cdot \sum_{j=1}^b d_j} \prod_{j=1}^bt_{\gamma_j}^{d_j} \right) \phi$$ which, by \cref{eq:sufficient}, has positive fractional Dehn twist coefficients. Also by the same reasoning as in \Cref{lem:essential_part} $$\mathrm{sc}(\calB^{-kd_j}t_\gamma^{d_j}\phi, \gamma) = \mathrm{sc}(\phi, \gamma)+ \beta_j d_j>0$$ because $d_j-1$ was the smallest positive number such that $\mathrm{sc}(\calB^{-k(d_j-1)}t_\gamma^{d_j-1}\phi, \gamma)$ has the same sign as the corresponding screw number of $\phi$. Therefore, $\calB^{-kd_j}t_\gamma^{d_j}\phi$ has a positive screw number at the orbit of $\gamma$. 
	
	Hence, by our main theorem \Cref{thm:positive_factorization}, the homeomorphism $\hat{\phi}$ admits a positive factorization. By \cref{eq:pos_fact_boun}, the homeomorphism $\calB^{k d_j} t_{\gamma_j}^{-d_j}$ also admits a positive factorization. Putting these two last facts together, we get that $$ \left( \calB^{k \cdot \sum_{j=1}^b d_j} \prod_{j=1}^bt_{\gamma_j}^{-d_j} \right)\hat{\phi}= \phi$$ admits a positive factorization.
\end{proof}

\begin{remark}\label{rem:cons_non_separating}
	Observe that there is a somewhat strong constraint in \Cref{thm:criterion} which is that we can only give such quantitative results when we are {\em correcting} negative Dehn twists around {\em non-separating} simple closed curves. This is because \Cref{thm:positive_fact_boundary} works with this restriction.
	
	 Observe also that \Cref{prop:pos_fact} implies that a left-handed Dehn twist in $\pmodgr$ around a separating simple closed curve admits a positive factorization. The problem is that we do not know how many powers of boundary twists this factorization produces in $\modgr$. An answer to that question would give a much stronger version of \Cref{thm:criterion}. In the language introduced  in this section, this problem can be summarized in the following important question:
	 
	 \textbf{Question 1.} Let $t_\gamma^{-1}$ be a left-handed Dehn twist around a separating simple closed curve $\gamma$ in $\Si_{g,r}$. This curve $\gamma$ is characterized, up to conjugacy in $\modgr$, by the two surfaces $\Si_{g_0,r_0}$ and $\Si_{g_1,r_1}$ in which it separates the original surface. How to express the correcting exponent (\Cref{def:correcting}) $N_{t_\gamma^{-1}}$ as a function of $g_0,r_0,g_1$ and $r_1$?
	 
\end{remark}

We finish the article by introducing a natural refinement of the correcting exponent which could be useful for future exploration of the set $\veer \setminus \dehn$.

	Let $\phi \in \modgr$ and let $\hat{\phi}$ be its image in $\pmodgr$. We define  $$\Delta_\phi:=\left\{ \psi \in \modgr : \hat \psi = \hat \phi \text{ in $\pmodgr$ and $\psi$ admits a positive factorization}\right\}.$$ We know that this is non-empty in the majority of cases by \Cref{cor:obstruc_dehn}. Index the boundary components  $B_1, \ldots, B_r$ of $\Si_{g,r}$ and denote by $\fr(\phi) \in \QQ^r$ the corresponding vector of fractional Dehn twist coefficients. 

\begin{definition}\label{def:correcting_semigroup}
	Let $\phi \in \modgr$ be pseudoperiodic. The {\bf correcting poset} of $\phi$ is: $$L_{\phi}:= \left\{(a_1, \ldots, a_r) \in \ZZ^r : (a_1, \ldots, a_r) = \fr(\psi \circ \phi^{-1}) \text{ for some } \psi \in \Delta_\phi\right\}.$$ 
\end{definition}

We observe that $L_{\phi}$ is well-defined since $\psi \circ \phi^{-1}$ is the identity outside a collar neighborhood of the boundary and the rational parts of $\fr(\psi)$ and $\fr(\phi)$ coincide because their projection in $\pmodgr$ coincide, that is, $\hat{\psi} = \hat{\phi}$ and the rational part of the fractional Dehn twist coefficient of $\phi$ at $B_i$ coincides with the rotation number of a periodic representative $\hat{\phi} \in \pmodgr$ at $B_i$.

Next we prove some easy properties of correcting posets.

\begin{prop}
Let $\phi \in \modgr$ be pseudoperiodic. Then $L_\phi$ satisfies the following properties:
\begin{enumerate}
	\item \label{it:p_i} It is a partially-ordered set with the relation $(a_1, \ldots, a_r) \leq (a_1', \ldots, a_r)'$ if and only if $a_i \leq a_i'$ for all $i=1, \ldots, r$.
	\item \label{it:p_ii} $\phi$ admits a positive factorization if and only if $(0, \ldots, 0) \in L_\phi$.
	\item \label{it:p_iii} $L_{\tilde{\phi}} \subset L_\phi$.
	\item \label{it:p_iv} If $(a_1, \ldots, a_r) \in L_\phi$ then $(a_1, \ldots,a_i+1, \ldots a_r) \in L_{\phi}$ for any $i=1, \ldots, r$.
\end{enumerate}
\end{prop}

\begin{proof}
	The first property \cref{it:p_i} is clear by definition. We do not use the lexicographic order (which would induce a total order) because the index of the boundary components is arbitrary.

	The second property \cref{it:p_ii} follows from the definition of $\Delta_{\phi}$: $(0, \ldots, 0) \in L_{\phi} \Leftrightarrow$ there exists $\psi \in \Delta_{\phi}$ such that $\fr(\psi \circ \phi^{-1})=(0, \ldots, 0) \Leftrightarrow \psi = \phi$ in $\modgr$.
	
	The inclusion of \cref{it:p_iii} follows from the fact that $\fr(\tilde\phi,B_i) \leq \fr(\phi,B_i)$ for all boundary components and from the observation that $\Delta_\phi= \Delta_{\tilde{\phi}}$ (because $\hat{\phi} = \hat{\tilde{\phi}}$).
	
		Let $e_i:=(\overbrace{0, \ldots, 0}^{i-1},1, 0, \ldots,0)$. In order to prove \cref{it:p_iv} we show that for any $(a_1, \ldots, a_r) \in L_{\phi}$ and any $i=1, \ldots, r$, the point $(a_1, \ldots, a_r)+ e_i$ is in $L_{\phi}$. Let $\psi \in \modgr$ be such that $\fr (\psi \circ \phi^{-1})=(a_1, \ldots, a_r)$, then  $t_{B_i}\circ\psi \in \Delta_\phi$ because $\widehat{t_{B_i}\circ\psi} = \hat{t}_{B_i} \circ \hat{\psi} =  \hat{\psi}$.
\end{proof}

We finish the article with two questions about correcting posets.

\textbf{Question 2.} Take the convex hull of  $L_\phi$ in $\RR^n$. How many compact faces does this  polyhedron have?

\textbf{Question 3.} Consider the class of posets $\mathcal{L}^s$ that come from monodromies associated with reduced holomorphic map germs defined on isolated complex surface singularities (including plane curve singularities). Consider also the class of posets $\mathcal{L}^+$ corresponding to mapping classes that admit a positive factorization in $\modgr$. The main \Cref{thm:positive_factorization} implies $\mathcal{L}^s \subset \mathcal{L}^+$. Is there an invariant of the poset $\mathcal{L}^+$ that detects elements in $\mathcal{L}^+ \setminus \mathcal{L}^s$?

\bibliographystyle{alpha}
\bibliography{bibliography}

\end{document}